\documentclass[11pt,a4paper]{article}
\usepackage[cp1251]{inputenc}
\usepackage{amsmath,amsthm}
\usepackage{amsfonts,amstext,amssymb,verbatim,epsfig}
\usepackage{dsfont}
\usepackage{enumerate}
\usepackage{psfrag}
\usepackage{graphicx}

\newtheorem{theorem}{Theorem}[section]
\newtheorem{corollary}[theorem]{Corollary}
\newtheorem{lemma}[theorem]{Lemma}
\newtheorem{proposition}[theorem]{Proposition}

\newtheoremstyle{likedef}
  {}%
  {}%
  {}%
  {\parindent}%
  {\bfseries}%
  {.}%
  {.5em}%
  {}%

\theoremstyle{likedef}
\newtheorem{definition}[theorem]{Definition}
\newtheorem{claim}[theorem]{Claim}

\numberwithin{equation}{section}

\begin{document}
\title{A short proof of the phase transition for \\ the vacant set of random interlacements}

\author{
Bal\'azs R\'ath\thanks{
Budapest University of Technology, Institute of Mathematics,
MTA-BME Stochastics Research Group, 1 Egry J\'ozsef u., 1111 Budapest, Hungary.
\noindent e-mail: \texttt{rathb@math.bme.hu}
}
}

\maketitle


\begin{abstract}
The vacant set of random interlacements at level $u>0$, introduced in \cite{SznitmanAM}, is a 
percolation model on $\mathbb{Z}^d$, $d \geq 3$ which arises as the set of sites avoided by a Poissonian cloud of 
doubly infinite trajectories, where $u$ is a parameter controlling the density of the cloud. 
It was proved in \cite{ SidoraviciusSznitman_RI, SznitmanAM} that for any $d \geq 3$
 there exists a positive and finite threshold $u_*$ such that if $u<u_*$ then the vacant set percolates and if
 $u>u_*$ then the vacant set does not percolate.
  We give an elementary proof of these facts. Our method also gives simple upper and lower bounds on the value of $u_*$
   for any $d \geq 3$.

\end{abstract}

\section{Introduction}

The model of random interlacements was introduced in \cite{SznitmanAM}.  
The interlacement $\mathcal{I}^u$ at level $u>0$  is a random subset of $\mathbb{Z}^d$, $d\geq 3$ 
that arises as the local  limit  as $N \to \infty$ of the range of the first $\lfloor u  N^d \rfloor$ steps of a simple random walk on
the discrete torus $(\mathbb{Z}/N\mathbb{Z})^d$, $d \geq 3$, see \cite{windisch_torus}.  
The law of $\mathcal I^u$ is characterized by 
\begin{equation}\label{def_eq:Iu_capa}
\mathbb P[\mathcal I^u\cap K = \emptyset] = e^{-u\cdot \mathrm{cap}(K)}, \quad\text{for any finite $K\subseteq\mathbb{Z}^d$,}
\end{equation}
where $\mathrm{cap}(K)$ denotes the discrete capacity of $K$, see \eqref{def_eq_capacity}.
The vacant set of random interlacements $\mathcal V^u$ at level $u$ is defined as the complement of $\mathcal I^u$ at level $u$:
\begin{equation}\label{def:vsri}
\mathcal V^u = \mathbb{Z}^d \setminus \mathcal I^u,\quad u>0 .\
\end{equation}
By \cite[(1.68)]{SznitmanAM} the correlations of $\mathcal V^u$ decay polynomially for any $u>0$:
\begin{equation}\label{eq:ri:correlations}
 \mathbb P[x,y\in\mathcal V^u] -
 \mathbb P[x\in\mathcal V^u]\cdot \mathbb P[y\in\mathcal V^u] \asymp 
 (|x-y| \vee 1 )^{2-d}~,\qquad x,y\in\mathbb{Z}^d.
\end{equation}
One is interested in the connectivity properties of the subgraphs of the nearest-neighbour lattice $\mathbb{Z}^d$ 
spanned by the above random sets. For any $u>0$, $\mathcal I^u$ is a $\mathbb{P}$-a.s.\  connected random subset of $\mathbb{Z}^d$ (see \cite[(2.21)]{SznitmanAM}),
but $\mathcal V^u$ exhibits a percolation phase transition:
there exists $u_*\in(0,\infty)$ such that 
\begin{itemize}\itemsep0pt
\item[(i)] 
for any $u>u_*$, $\mathbb{P}$-a.s.\ all connected components of $\mathcal V^u$  are finite, and 
\item[(ii)]
for any $u<u_*$, $\mathbb{P}$-a.s.\ $\mathcal V^u$ contains an infinite connected component. 
\end{itemize}
The fact that $u_*<\infty$ was proved in \cite[Section 3]{SznitmanAM}, and the positivity of $u_*$ was established in \cite[Section 4]{SznitmanAM} when $d\geq 7$, 
and later in \cite{SidoraviciusSznitman_RI} for all $d\geq 3$. 

There is no reason to believe that an  exact formula for the value of the critical threshold $u_*=u_*(d)$ exists.
 However, it is proved in \cite{Sznitman_lower, Sznitman_upper} that
 \begin{equation}\label{u_star_high_dim_sznitman}
 \lim_{d \to \infty} \frac{u_*(d)}{ \ln(d)}=1,
 \end{equation}
  in agreement with the principal asymptotic behaviour of the critical threshold of random interlacements on $2d$-regular trees, which is explicitly computed in \cite[Proposition 5.2]{teixeira_EJP_weighted}.

\medskip 
 

 

The aim of this paper is to give a short proof of the non-triviality of phase transition of $\mathcal{V}^u$
 and to provide
simple explicit upper and lower bounds on the value of $u_*=u_*(d), d \geq 3$.

\medskip

 For any $d \geq 3$ let us denote by $0<c_g=c_g(d)$ and $C_g=C_g(d) < +\infty $ the best constants such that the inequalities
\begin{equation}\label{green_bounds}
c_g \cdot (\vert x-y \vert \vee 1)^{2-d} \leq  g(x,y) \leq C_g \cdot (\vert x-y \vert \vee 1)^{2-d}, \qquad x,y \in \mathbb{Z}^d
\end{equation}
hold, where $| \cdot |$ is the $\ell^\infty$-norm  on $\mathbb{Z}^d$ and
 $g(\cdot,\cdot)$ is the Green function of simple random walk on $\mathbb{Z}^d$, see \eqref{def_eq_green}.
The positivity of $c_g$ and $C_g<+\infty$  follow from  \cite[Theorem 1.5.4]{L91}.

\begin{theorem}\label{thm_bounds_on_u_star}
 For any $d \geq 3$, we have
\begin{equation}\label{eq_thm_bounds_on_u_star}
   \frac{c_g}{L_0}  \frac{1}{\mathcal{C}_2} 2^{-(d+5)}  \leq u_* \leq \frac{5}{2} C_g \ln(\mathcal{C}_d), 
   \end{equation}
where 
\begin{equation}\label{def_eq_C_d}
 \mathcal{C}_d=(13^d-11^d)(25^d-23^d), \quad d \geq 2,
 \end{equation}
and 
 \begin{equation}\label{def_eq_L_0_supercrit}
L_0= \left\{
\begin{array}{ll}
 \left\lceil \exp\left( 48 \frac{C_g}{c_g} \mathcal{C}_2 \right) \right\rceil
& \text{ if } \quad d=3,\\
  & \\
 \left\lceil \left( 48 \frac{C_g}{c_g} \mathcal{C}_2 \right)^{\frac{1}{d-3}} \right\rceil 
& \text{ if } \quad d \geq 4. 
\end{array} \right.
\end{equation}
  \end{theorem}

The bounds \eqref{eq_thm_bounds_on_u_star} are not at all sharp, especially if we compare them with \eqref{u_star_high_dim_sznitman} as $d \to \infty$. 
This shortcoming of Theorem \ref{thm_bounds_on_u_star} is counterbalanced by the fact that its proof is 
very simple. In particular, our self-contained proof does not use the
 \emph{``sprinkling'' technique} and \emph{decoupling inequalities}  usually applied in order to overcome the 
 long-range correlations \eqref{eq:ri:correlations} present in the model. 
 The proof of $u_*(d)>0$ for $d \geq 7$ in \cite[Section 4]{SznitmanAM} does not use ``sprinkling'', but the 
 proof of $u_*(d)<+\infty$ for any $d \geq 3$ in \cite[Section 3]{SznitmanAM} and the proof of $u_*(d)>0$ for $3 \leq d \leq 7$  
 in \cite{SidoraviciusSznitman_RI} does. Various forms of 
 decoupling inequalities have been subsequently developed to study the connectivity properties of $\mathcal{V}^u$ in 
 the subcritical \cite{PoTe, SidSzn_aihp, Sznitman:Decoupling} 
 and supercritical \cite{ DRS,  Teixeira} phases. 
 These  techniques are very useful once they are available, but the elementary method of our paper
 seems to be easier to adapt to other percolation models with long-range correlations, 
 e.g., \emph{branching interlacements} \cite{branching}.

\medskip

Let us briefly describe the idea of the proof of Theorem \ref{thm_bounds_on_u_star}. 
We employ multi-scale renormalization.
In order to prove $u_*<+\infty$ we show that if $\mathcal{V}^u$ crosses an annulus at scale $L_n=6^n$ 
then this vacant crossing contains a set $\mathcal{X}_{\mathcal{T}}$ of $2^n$ well-separated  vertices which arises as the image of leaves
 under an 
 embedding $\mathcal{T}$  of the dyadic tree  of depth $n$ 
 (this method already appears in \cite{Sznitman:Decoupling}). 
 By construction, the number of possible embeddings is less than
 $\mathcal{C}_d^{2^n}$ (c.f.\ \eqref{def_eq_C_d}), so we only need to show that $\mathrm{cap}(\mathcal{X}_{\mathcal{T}}) \asymp 2^n$ 
 if we want to use \eqref{def_eq:Iu_capa} to to show that crossing of the annulus by $\mathcal{V}^u$ is unlikely when $u$ is big enough. 
 This is indeed the case, because
  by construction the embedding $\mathcal{T}$ is ``spread-out on all scales", 
  thus the cardinality and the capacity of $\mathcal{X}_{\mathcal{T}}$ are comparable.
  
  In order to prove $u_*>0$, we restrict our attention to a plane inside $\mathbb{Z}^d$. By planar duality we only
  need to show that a $*$-connected crossing of a planar annulus at scale $L_n=L_0 \cdot 6^n$ 
  by $\mathcal{I}^u$ is unlikely.
  We show that such a crossing must intersect $2^n$ ``frames'', where each frame is the union of four ``sticks'' of length 
  $2 L_0-1$. Such a collection of frames again arises from a spread-out embedding of the  dyadic tree  of depth $n$.
We use that $\mathcal{I}^u$ can be written as the union of the ranges of a Poissonian cloud of independent
 random walks and the fact that random walks tend to avoid sticks if $L_0$ is large enough 
 (c.f.\ \eqref{def_eq_L_0_supercrit})  to arrive at a large deviation
 estimate on the probability that the number of frames that intersect 
$\mathcal{I}^u$ is  $2^n$ which is strong enough to beat the combinatorial complexity term $\mathcal{C}_2^{2^n}$. 
This stick-based approach to $u_*>0$ is already present in \cite[Section 3]{SidoraviciusSznitman_RI} and our
 large deviation estimate resembles the one in the proof of \cite[Theorem 2.4]{SznitmanAM}.

\medskip

The rest of this paper is organized as follows.

In Section \ref{section:preliminaries} we introduce further notation and
recall some useful facts related to the notion of capacity and  random interlacements.
In Section \ref{section:renormalization} we define the notion of a \emph{proper embedding} 
of a dyadic tree into $\mathbb{Z}^d$ and derive some facts about such embeddings. 
In Sections \ref{section:upper} and \ref{section:lower} we prove the upper and lower bounds on $u_*$ stated in Theorem \ref{thm_bounds_on_u_star}.

\section{Preliminaries} \label{section:preliminaries}
For a set $K$, we denote by $\vert K\vert$ its cardinality. 
We denote by $K \subset \subset \mathbb{Z}^d$ the fact that $K$ is a finite subset of $\mathbb{Z}^d$.
 We denote by $| x |$ the $\ell^\infty$-norm of $x \in \mathbb{Z}^d$ and by  $S(x,R)$ the $\ell^\infty$-sphere of radius $R$ about $x$ in $\mathbb{Z}^d$:
\begin{equation} \label{sphere}
 S(x,R)=\{ y \in \mathbb{Z}^d\, : \, |y-x| =R \}.
\end{equation}

  For $x\in\mathbb{Z}^d$, denote by $P_x$ the law of simple random walk 
  $\left(X_n\right)_{n=0}^{\infty}$ on $\mathbb{Z}^d$  starting at $X_0=x$.
  If $m$ is a probability measure on $\mathbb{Z}^d$, we denote by
  \begin{equation}\label{rw_with_initial_dist_m}
   P_m= \sum_{x \in \mathbb{Z}^d} m(x)P_x
   \end{equation}
   the law of simple random walk with initial distribution $m$ and by $E_m$ the corresponding expectation.
    The Green function  of simple random walk on $\mathbb{Z}^d$ is defined by
\begin{equation}\label{def_eq_green}
g(x,y) = \sum_{n = 0}^{\infty} P_x [X_n = y], \quad x, y \in \mathbb{Z}^d.
\end{equation}
Let us denote by $\{X\} \subseteq \mathbb{Z}^d$ the range of the random walk:
\begin{equation}\label{def_eq_range}
 \{X\} = \cup_{n=0}^{\infty} \{ X_n \}
 \end{equation}

\subsection{Potential theory}\label{subsection_potential_theory}

If $K \subset \subset \mathbb{Z}^d$,  we define the equilibrium measure $e_K(\cdot)$ of $K$ by  
\begin{equation*}
e_K(x) = P_x [\, X_n \notin K \text{ for any } n \geq 1 \, ], \qquad x \in K.
\end{equation*}

The total mass of the equilibrium measure is called the capacity of $K$:
\begin{equation}\label{def_eq_capacity}
\mathrm{cap}(K) = \sum_{x \in K} e_K(x).
\end{equation}

One defines the normalized equilibrium measure $\widetilde{e}_K(\cdot)$ of $K$ by
\begin{equation}\label{def_eq_normalized_eq_measure}
\widetilde{e}_K(x) =  \frac{e_K(x)}{\mathrm{cap}(K)}.
\end{equation}

Let us now collect some facts about capacity that we will use in the sequel. The proofs of the properties
\eqref{green_equilib_entrance_identity}-\eqref{bound_on_capacity} below can be found in, e.g., \cite[Section 1.3]{DRS_book}.

For any $x\in\mathbb{Z}^d$ and any $K \subset \subset \mathbb{Z}^d$, 
\begin{equation}\label{green_equilib_entrance_identity}
 P_x [ \{X\} \cap K \neq \emptyset ]  = 
\sum_{y \in K} g(x,y) e_K(y) \stackrel{ \eqref{def_eq_capacity} }{\leq} \mathrm{cap}(K) \max_{y \in K} g(x,y) .
\end{equation}

For any $K_1,K_2 \subset \subset \mathbb{Z}^d$,
\begin{equation}\label{subadditive}
\mathrm{cap}(K_1 \cup K_2) \leq \mathrm{cap}(K_1)+ \mathrm{cap}(K_2).
\end{equation}

For any $K \subseteq K' \subset \subset \mathbb{Z}^d$,
\begin{equation}\label{capa_monotone}
\mathrm{cap}(K) \leq \mathrm{cap}(K').
\end{equation} 

For any $K \subset \subset \mathbb{Z}^d$,
\begin{equation}\label{bound_on_capacity}
 \frac{|K|}{\max_{x \in K} \sum_{y \in K}{g(x,y)}}
\leq \mathrm{cap}(K) \leq \frac{|K|}{\min_{x \in K} \sum_{y \in K}{g(x,y)}}.
\end{equation}

Let us denote by $F$ the plane 
\begin{equation}\label{plane}
 F = \mathbb{Z}^2\times \{0\}^{d-2} \subseteq \mathbb{Z}^d .
 \end{equation}
 
For any $y \in F$ and $L \geq 1$ let us define the frame $\Box_y^L \subseteq F$ by
\begin{equation}\label{def_eq_frame_1}
\Box_y^L \stackrel{ \eqref{sphere} }{=} S(y, L -1) \cap F. 
\end{equation}

The next lemma gives an explicit upper bound on the capacity of a frame.
The bounds of \eqref{capa_frame_bound} are actually sharp up to a dimension-dependent constant factor, but we will only 
use the upper bounds.
The stronger bound for $d=3$ is  crucial to showing that 
random walks tend to avoid frames in $\mathbb{Z}^3$. The extra $\ln(L_0)$ makes the
 parameter $p$ defined in \eqref{def_of_p_geo_paramater} small, which is necessary for
 our proof of $u_*(3)>0$.
 Recall the notion of $c_g$ from \eqref{green_bounds}.

\begin{lemma}\label{lemma_frame_capacity}
For any $L \geq 1$ we have
\begin{equation}\label{capa_frame_bound}
\mathrm{cap} \left( \Box_y^L \right) \leq
\begin{cases} 
8\frac{L}{c_g} & \text{ if } \qquad d \geq 4, \\
8\frac{L}{c_g \cdot (1+\ln(L))} & \text{ if } \qquad d = 3.
\end{cases}
\end{equation}
\end{lemma}

\begin{proof}
Denote by 
$\mathcal{S}_\ell=\{ 1,\dots, \ell\} \times \{0\}^{d-1} \subseteq \mathbb{Z}^d$
 the stick of length $\ell$. We will use  \eqref{bound_on_capacity} to bound $\mathrm{cap}(\mathcal{S}_\ell)$.
 If $x \in \mathcal{S}_\ell$ then $x=\{i\} \times \{0\}^{d-1}$ for some $1\leq i \leq \ell$ and
\begin{multline*}
\sum_{y \in \mathcal{S}_\ell}{g(x,y)} 
\stackrel{ \eqref{green_bounds} }{\geq}
 \sum_{j=1}^{\ell} c_g \cdot (|j-i| \vee 1)^{2-d} \geq 
 \sum_{j=1}^{\ell} c_g \cdot (|j-1| \vee 1)^{2-d}=\\
c_g \cdot \left( 1+ \sum_{k=1}^{\ell-1} k^{2-d} \right) \geq
\begin{cases}
c_g & \text{ if } \quad d \geq 4,\\
c_g \cdot \left(1+ \int_1^{\ell} \frac{1}{s} \, \mathrm{d} s \right) 
= c_g\cdot (1+ \ln(\ell)) & \text{ if } \quad d =3.
\end{cases}
\end{multline*}
Using these bounds,  \eqref{bound_on_capacity} and $|\mathcal{S}_\ell|=\ell$
 we obtain that $\mathrm{cap} \left( \mathcal{S}_\ell \right) \leq \ell /c_g$ if $d \geq 4$ and
 $\mathrm{cap} \left( \mathcal{S}_\ell \right) \leq \ell /(c_g \cdot (1+\ln(\ell)) )$ if $d =3$.
 Now the frame $\Box_y^L$ is the union of four sticks of length $2L-1$, thus 
 \eqref{capa_frame_bound} follows from the above bounds and \eqref{subadditive}, \eqref{capa_monotone}.
 
\end{proof}

\subsection{Constructive definition of random interlacements}

The definition of the interlacement $\mathcal{I}^u$ at level $u$ by the formula \eqref{def_eq:Iu_capa} is short, but it is not constructive. 
The construction of \cite[Section 1]{SznitmanAM} involves a Poisson point process with intensity measure $u \cdot \nu,$ 
where  $\nu$ is a sigma-finite measure on the space of equivalence classes of
doubly infinite trajectories modulo time-shift.
 The union  of the ranges of trajectories which are contained in the support of this Poisson point process is denoted by
$\mathcal I^u$, and this random subset of $\mathbb{Z}^d$ indeed satisfies \eqref{def_eq:Iu_capa}.

\medskip 
We will not use the full definition of random interlacements, only a corollary of it, which 
allows one to construct a set with the same law as $\mathcal{I}^u \cap K$ for any $K \subset \subset \mathbb{Z}^d$.

Recall the notion of $P_m$ from \eqref{rw_with_initial_dist_m}, $\{X\}$ from \eqref{def_eq_range} and 
$\widetilde{e}_K(\cdot)$ from \eqref{def_eq_normalized_eq_measure}.

\begin{claim}\label{claim:constructive_interlacement}
Let $d\geq 3$, $K \subset \subset \mathbb{Z}^d$, $N_K$ be a Poisson random variable with parameter $u\cdot \mathrm{cap}(K)$, and 
$(X^j)_{j\geq 1}$ i.i.d.\ simple random walks with distribution $P_{\widetilde e_K}$ and independent from $N_K$. 
Then  $K \cap \cup_{j=1}^{N_K}\{X^j\} $
has the same distribution as $\mathcal I^u\cap K$. 
\end{claim}
This explicit ``local representation" of $\mathcal{I}^u$
follows from the very construction of  the sigma-finite measure
 $\nu$, which is obtained by patching
together certain explicit measures $Q_K$, $K \subset \subset \mathbb{Z}^d$ in a consistent manner in \cite[Theorem 1.1]{SznitmanAM}.
The above representation of $\mathcal I^u\cap K$  is obtained from 
the Poisson point process with intensity measure $u Q_K$.

\section{Renormalization}\label{section:renormalization}

For $n\geq 0$, let $T_{(n)}=\{1,2\}^n$ (in particular, $T_{(0)}=\emptyset$). Denote by  
\[ T_n =\bigcup_{k=0}^n T_{(k)} \]
the dyadic tree of depth $n$. 
For $0 \leq k < n$ and $m \in T_{(k)}$, $m=(\xi_1,\dots,\xi_k)$,  we denote
by 
\begin{equation}\label{def_eq_m1_m2}
 m_1=(\xi_1,\dots,\xi_k,1) \qquad \text{and} \qquad m_2=(\xi_1,\dots,\xi_k,2) 
 \end{equation}
the two children of $m$ in $T_{(k+1)}$. 
Given some $L_0\geq 1$ we define the sequence of scales
\begin{equation} \label{def:scalesLn}
L_n := L_0 \cdot 6^n , \quad n \ge 0.
\end{equation}
For $n \ge 0,$ we denote by $\mathcal L_n = L_n \mathbb{Z}^d$ the lattice $\mathbb{Z}^d$ renormalized by $L_n$.

\begin{definition}\label{def_proper_embedding_of_trees}
  $\mathcal{T}: T_n \to \mathbb{Z}^d$ is a proper embedding of  $T_n$  with 
root at $x \in \mathcal{L}_n$  if

\begin{enumerate}
 \item $\mathcal{T}(\emptyset)=x$;
\item for all $0 \leq k \leq n$ and $m \in T_{(k)}$ we have $\mathcal{T}(m) \in \mathcal{L}_{n-k}$;
\item for all $0 \leq k < n$ and $m \in T_{(k)}$ we have
\begin{equation}\label{tree_children_spread_out}
|\mathcal{T}(m_1) -\mathcal{T}(m)|= L_{n-k}, \qquad |\mathcal{T}(m_2) -\mathcal{T}(m)|= 2 L_{n-k}.
\end{equation}
\end{enumerate}
We denote by $\Lambda_{n,x}$ the set of proper embeddings of  $T_n$ into $\mathbb{Z}^d$ with root at $x$.
\end{definition}

\begin{lemma}\label{lemma_cardinality_of_embeddings}
For any $L_0 \geq 1$,  $n \geq 0$ and  $x \in \mathcal{L}_n$ the number of  proper embeddings of  $T_n$ into $\mathbb{Z}^d$ with 
root at $x$ is equal to
\begin{equation}\label{lambda_n_x_cardinality}
|\Lambda_{n,x}|\stackrel{\eqref{def_eq_C_d}}{=}\mathcal{C}_d^{2^{n}-1} .
\end{equation}
\end{lemma}
\begin{proof} The claim is trivially true for $n=0$.
If $n \geq 1$, $x \in \mathcal{L}_n$ and $\mathcal{T} \in \Lambda_{n,x}$, we denote by $\mathcal{T}_1$ and $\mathcal{T}_2$ the
 two embeddings of $T_{n-1}$ which arise 
from $\mathcal{T}$ as the embeddings of the descendants of the two children of the root, i.e., 
for any $0 \leq k \leq n-1$ and $m=(\xi_1,\dots,\xi_k) \in T_{(k)}$ let
$\mathcal{T}_\xi(m)=\mathcal{T}(\xi,\xi_1, \xi_2, \ldots, \xi_k)$ for $\xi \in \{1,2\}$. 
By Definition \ref{def_proper_embedding_of_trees} we have $\mathcal{T}_\xi \in \Lambda_{n-1,\mathcal{T}(\xi)}$ for $\xi \in \{1,2\}$, thus we obtain \eqref{lambda_n_x_cardinality} by induction on $n$:
\begin{multline*}\label{eq_bound_on_number_of_trees}
|\Lambda_{n,x}| \stackrel{\eqref{tree_children_spread_out}}{=}
 | S(x, L_n) \cap \mathcal{L}_{n-1}| \cdot | S(x, 2 L_n) \cap \mathcal{L}_{n-1}|\cdot 
|\Lambda_{n-1,\mathcal{T}(1)}|\cdot |\Lambda_{n-1,\mathcal{T}(2)}|\stackrel{\eqref{def:scalesLn}}{=}\\
| S(0, 6)| \cdot | S(0, 12) |\cdot 
|\Lambda_{n-1,\mathcal{T}(1)}|\cdot |\Lambda_{n-1,\mathcal{T}(2)}|
  \stackrel{ (*) }{=} 
\mathcal{C}_d \cdot \mathcal{C}_d^{2^{n-1}-1} \cdot \mathcal{C}_d^{2^{n-1}-1} =\mathcal{C}_d^{2^{n}-1},
\end{multline*}
where in $(*)$ we used  the induction hypothesis.
\end{proof}

We say that $\gamma: \{0,\dots, l\} \to \mathbb{Z}^d$ is a $*$-connected path if 
$|\gamma(i)-\gamma(i-1)|=1$ for any $1 \leq i \leq l$. For such a path
 we denote by $ \{\gamma\}= \{\gamma(1), \dots, \gamma(l)\}$ the range of $\gamma$.
 
Recall the notion of  $S(x,R)$ from \eqref{sphere} and note that
$S(x,0)=\{x\}$.

\begin{lemma}\label{lemma_path_leaves_spheres_intersect}
If $\gamma$ is a $*$-connected path in $\mathbb{Z}^d$, $d\geq 2$  and $x \in \mathcal{L}_n$ such that
\begin{equation}\label{eq_annulus_crossing_path}
\{\gamma\} \cap S(x, L_n -1) \neq \emptyset  \quad \text{and} \quad
\{\gamma\} \cap S(x, 2L_n) \neq \emptyset
\end{equation}
then there exists $\mathcal{T} \in \Lambda_{n,x}$ such that 
\begin{equation}\label{leaves_spheres_path_intersect}
\{\gamma\} \cap S(\mathcal{T}(m), L_0 -1) \neq \emptyset \quad \text{for all} \quad m \in T_{(n)}.
\end{equation}
\end{lemma}

\begin{proof}
We will  prove that \eqref{eq_annulus_crossing_path} implies that 
there exists $\mathcal{T} \in \Lambda_{n,x}$ such that for all $0 \leq k \leq n$  we have
\begin{equation}\label{eq_annulus_crossing_path_induction}
\begin{array}{r}
\{\gamma\} \cap S(\mathcal{T}(m), L_{n-k} -1) \neq \emptyset \\ 
\{\gamma\} \cap S(\mathcal{T}(m), 2L_{n-k}) \neq \emptyset 
\end{array}
 \quad \text{for all} \quad m \in T_{(k)}.
\end{equation}
We will construct such a $\mathcal{T} \in \Lambda_{n,x}$ by induction on $k$. 
By $\mathcal{T}(\emptyset)=x$ we see that the case $k=0$ of \eqref{eq_annulus_crossing_path_induction}  is just \eqref{eq_annulus_crossing_path}.
Assuming that \eqref{eq_annulus_crossing_path_induction} holds for some $0 \leq k \leq n-1$
 we now show that it also holds for $k+1$. If $m \in T_{(k)}$ then our induction hypothesis 
 \eqref{eq_annulus_crossing_path_induction} and the fact that $\gamma$ is a $*$-connected path imply
\begin{equation*}
\begin{array}{r}
\{\gamma\} \cap S(\mathcal{T}(m), L_{n-k}+L_{n-k-1} -1) \neq \emptyset, \\
\{\gamma\} \cap S(\mathcal{T}(m), 2L_{n-k} -L_{n-k-1}+1 ) \neq \emptyset.
\end{array}
\end{equation*}
We also have
\begin{align*}
S(\mathcal{T}(m), L_{n-k}+L_{n-k-1} -1) &\subseteq 
\bigcup_{y \in S(\mathcal{T}(m), L_{n-k}) \cap \mathcal{L}_{n-k-1}} S(y, L_{n-k-1} -1),\\
S(\mathcal{T}(m), 2L_{n-k} -L_{n-k-1}+1 ) &\subseteq
\bigcup_{z \in S(\mathcal{T}(m), 2L_{n-k}) \cap \mathcal{L}_{n-k-1}} S(z, L_{n-k-1} -1),
\end{align*}
thus we can choose 
\[ \mathcal{T}(m_1) \in S(\mathcal{T}(m), L_{n-k}) \cap \mathcal{L}_{n-k-1} 
\quad \text{and} \quad
 \mathcal{T}(m_2)\in S(\mathcal{T}(m), 2L_{n-k}) \cap \mathcal{L}_{n-k-1}\]
  such that 
\begin{equation*}
\{\gamma\} \cap S(\mathcal{T}(m_1), L_{n-(k+1)} -1) \neq \emptyset, \qquad
\{\gamma\} \cap S(\mathcal{T}(m_2), L_{n-(k+1)} -1) \neq \emptyset.
\end{equation*}
It follows from this, $|\mathcal{T}(m_1)-\mathcal{T}(m_2)| \geq L_{n-k}= 6 L_{n-(k+1)}$  and the fact that $\gamma$ is a $*$-connected path that we also have
\begin{equation*}
\{\gamma\} \cap S(\mathcal{T}(m_1), 2L_{n-(k+1)}) \neq \emptyset, \qquad
\{\gamma\} \cap S(\mathcal{T}(m_2), 2L_{n-(k+1)}) \neq \emptyset.
\end{equation*}
We have thus constructed the embedding $\mathcal{T}$ up to depth $k+1$ so that 
Definition \ref{def_proper_embedding_of_trees} is satisfied up to depth $k+1$ and
 \eqref{eq_annulus_crossing_path_induction} also holds for $k+1$. 
 Therefore by induction we have constructed $\mathcal{T} \in \Lambda_{n,x}$ such that \eqref{eq_annulus_crossing_path_induction} holds for all $0 \leq k \leq n$, which implies
\eqref{leaves_spheres_path_intersect}. The proof of Lemma \ref{lemma_path_leaves_spheres_intersect} is complete.
\end{proof}

For $0 \leq k \leq n$ and $m=(\xi_1,\dots,\xi_n) \in T_{(n)}$ we denote  
$\left. m \right|_k = (\xi_1,\dots,\xi_k) \in T_{(k)}$.
Let us denote the lexicographic distance of $m, m' \in T_{(n)}$ by
\begin{equation*}
\rho(m, m')= \min \{ k \geq 0 \; : \; \left. m \right|_{n-k}=\left. m' \right|_{n-k} \}.
\end{equation*}
For any $m \in T_{(n)}$ and $0 \leq k \leq n$ we define
\begin{equation}\label{def_eq_tree_sphere}
T_{(n)}^{m,k}=\{ m' \in T_{(n)} \; : \; \rho(m,m')=k \},
\end{equation}
see Figure \ref{fig:canopy} for an illustration.  Note that
\begin{equation}\label{cardinality_tree_sphere}
|T_{(n)}^{m,k}|=2^{k-1}, \qquad  1 \leq k \leq n.
\end{equation}

\begin{figure}[h]
\begin{center}
\large

\psfrag{M}{$m$}
\psfrag{T1}{$T_{(n)}^{m,1}$}
\psfrag{T2}{$T_{(n)}^{m,2}$}
\psfrag{T3}{$T_{(n)}^{m,3}$}
\psfrag{E}{$\emptyset$}
\psfrag{1}{$1$}
\psfrag{2}{$2$}
\psfrag{11}{$11$}
\psfrag{12}{$12$}
\psfrag{21}{$21$}
\psfrag{22}{$22$}

  \includegraphics[scale=1]{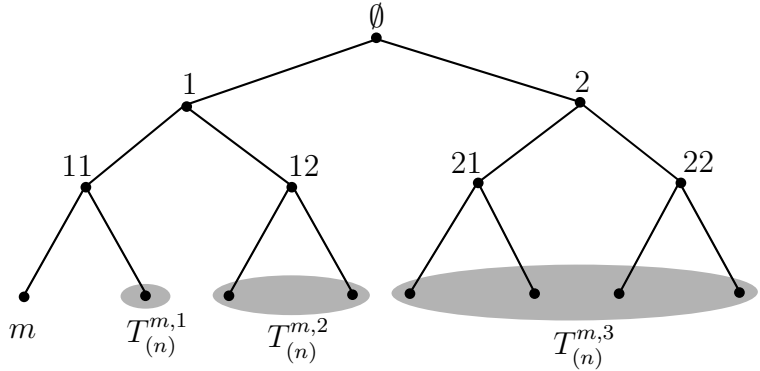}
\caption{An illustration of the subsets $T_{(n)}^{m,k}$ of leaves of $T_n$ defined in \eqref{def_eq_tree_sphere}.
The dyadic tree on the picture is of depth $n=3$ and the leaf denoted by $m$ is $111 \in T_{(n)}$.
}
\label{fig:canopy}
\end{center}
\end{figure}

The next lemma shows that a proper embedding  is ``spread-out on all scales."

\begin{lemma}\label{lemma_far_in_tree_far_in_embedding}

\begin{equation}\label{eq_far_in_tree_far_in_embedding}
\begin{array}{ccc}
\forall\;  n \geq 1,\; x \in \mathcal{L}_n,\;
 \mathcal{T} \in \Lambda_{n,x}, \; m \in T_{(n)}, \; k \geq 1, \; \\
\forall\; 
 m' \in T_{(n)}^{m,k}, \;
 y \in S(\mathcal{T}(m), L_0 -1), \; 
 z \in S(\mathcal{T}(m'), L_0 -1)
  : \\
 |y-z| \geq L_{k-1}.
 \end{array}
\end{equation}
\end{lemma}

\begin{proof}
Let $m''=\left. m \right|_{n-k}=\left. m' \right|_{n-k} \in T_{(n-k)}$.
 Recalling \eqref{def_eq_m1_m2} we may 
 assume w.l.o.g.\  that $\left. m \right|_{n-k+1}=m''_1 \in T_{(n-k+1)}$ and 
$\left. m' \right|_{n-k+1}=m''_2 \in T_{(n-k+1)}$.
We have 
\begin{equation*}
|\mathcal{T}(m''_1) - \mathcal{T}(m''_2)| \stackrel{ \eqref{tree_children_spread_out} }{\geq}
 L_k \stackrel{\eqref{def:scalesLn} }{=}6 L_{k-1},
\end{equation*}
moreover
\begin{multline*}
|\mathcal{T}(m''_1) - y| \leq |\mathcal{T}(m)-  y| +
\sum_{j=1}^{k-1} \left| \mathcal{T}(\left. m \right|_{n-j})-\mathcal{T}(\left. m \right|_{n-j+1}) \right| 
\stackrel{ \eqref{tree_children_spread_out} }{\leq} \\
L_0-1+ \sum_{j=1}^{k-1} 2 L_j 
\stackrel{ \eqref{def:scalesLn} }{\leq} 2 L_{k-1} \sum_{i=0}^{\infty} 6^{-i}=
\frac{12}{5} L_{k-1},
\end{multline*}
and similarly $|\mathcal{T}(m''_2) - z| \leq \frac{12}{5} L_{k-1}$.
Putting  these bounds  together we obtain \eqref{eq_far_in_tree_far_in_embedding}.
\end{proof}

\section{Upper bound on $u_*$}\label{section:upper}

Let us choose $L_0=1$ in \eqref{def:scalesLn}.
For $n \geq 1$  let us denote by $A^u_{n}$ the event 
\[ A^u_{n} = \left\{ \begin{array}{cc}
 \text{ there exists a nearest-neighbour path in $\mathcal{V}^u$ } \\
\text{ that connects $S(0, L_n-1)$ to $S(0,2 L_n)$ }
\end{array}
 \right\}. \]
Recall the definitions of  $C_g$ from \eqref{green_bounds} and $\mathcal{C}_d$ from \eqref{def_eq_C_d}.
\begin{proposition}\label{prop_subcrit}
For any $d \geq 3$ and 
\begin{equation}\label{def_eq_subcrit_u_upper_bound_prop}
 u> \frac{5}{2} C_g \ln(\mathcal{C}_d)
 \end{equation}
there exists $q=q(d,u) \in (0,1)$ such that
for any $n \geq 1$  we have 
\begin{equation}\label{annulus_cross_with_tiny_prob}
\mathbb{P}[ A^u_{n} ] \leq q^{2^n}.
\end{equation}
\end{proposition}

\begin{corollary} Proposition \ref{prop_subcrit} implies the upper bound of Theorem \ref{thm_bounds_on_u_star},
 as we now explain.
Let us denote by $\widetilde{A}^u_{n}$ the event that there exists a nearest-neighbour path in $\mathcal{V}^u$ that connects $S(0, L_n-1)$ to infinity and by 
$\widetilde{A}^u_{\infty}$ the event that $\mathcal{V}^u$ has an infinite connected component. 
If \eqref{def_eq_subcrit_u_upper_bound_prop} holds, then
\begin{equation*}
 \mathbb{P}[ \widetilde{A}^u_{\infty} ] \stackrel{(*)}{=}
 \lim_{n \to \infty}\mathbb{P}[ \widetilde{A}^u_{n} ] \leq 
  \lim_{n \to \infty}\mathbb{P}[ A^u_{n} ] \stackrel{ \eqref{annulus_cross_with_tiny_prob} }{=}0,
  \end{equation*}
  where $(*)$ holds by monotone convergence. Therefore we have $u_* \leq \frac{5}{2} C_g \ln(\mathcal{C}_d)$.
\end{corollary}

\begin{proof}[Proof of Proposition \ref{prop_subcrit}]
For any $n \geq 1$ and $\mathcal{T} \in \Lambda_{n,0}$ we denote
$ \mathcal{X}_{\mathcal{T}}= \bigcup_{m \in T_{(n)}} \mathcal{T}(m).$
Noting that $S(\mathcal{T}(m), L_0 -1)=S(\mathcal{T}(m), 0)=\{ \mathcal{T}(m)\}$ for any $m \in T_{(n)} $ and that 
every nearest-neighbour path is also a $*$-connected path we can apply Lemma \ref{lemma_path_leaves_spheres_intersect}  to infer 
 \begin{multline}\label{subcrit_tree_union_bound}
  \mathbb{P} \left[ A^u_{n} \right] 
  \stackrel{ \eqref{leaves_spheres_path_intersect} }{\leq}
\mathbb{P} \left[ \bigcup_{\mathcal{T} \in \Lambda_{n,0}} \{ \mathcal{X}_{\mathcal{T}} \subseteq \mathcal{V}^u \} \right] 
\stackrel{ \eqref{def_eq:Iu_capa}, \eqref{def:vsri} }{\leq} \\
 \sum_{\mathcal{T} \in \Lambda_{n,0}} \exp \left(-u \cdot \mathrm{cap}(\mathcal{X}_{\mathcal{T}}) \right) \stackrel{\eqref{lambda_n_x_cardinality}  }{\leq} 
\mathcal{C}_d^{2^n} \cdot \max_{\mathcal{T} \in \Lambda_{n,0}}  \exp \left(-u \cdot \mathrm{cap}(\mathcal{X}_{\mathcal{T}}) \right).
 \end{multline}
In order to finish the proof of Proposition \ref{prop_subcrit} we only need to show
 that for any $\mathcal{T} \in \Lambda_{n,0}$ we have
\begin{equation}\label{subcrit_capa_bound}
 \mathrm{cap}(\mathcal{X}_{\mathcal{T}}) \geq \frac{2}{5} \frac{1}{C_g}  2^n,
\end{equation}
because then  we indeed obtain
\[ \mathbb{P} \left[ A^u_{n} \right] 
\stackrel{ \eqref{subcrit_tree_union_bound}, \eqref{subcrit_capa_bound} }{\leq}
\mathcal{C}_d^{2^n} \exp \left(-u \frac{2}{5} \frac{1}{C_g}  2^n  \right)= 
\left(  \mathcal{C}_d \exp \left(-u \frac{2}{5} \frac{1}{C_g}   \right) \right)^{2^n} 
= q^{2^n}, \qquad q \stackrel{ \eqref{def_eq_subcrit_u_upper_bound_prop}}{<}1.
\]
 We will show \eqref{subcrit_capa_bound} using \eqref{bound_on_capacity}.
For any $\mathcal{T} \in \Lambda_{n,0}$ and any $m \in T_{(n)}$ we have
\begin{multline}\label{subcrit_upper_bound_on_green_sum}
 \sum_{m' \in  T_{(n)}} g(\mathcal{T}(m), \mathcal{T}(m')) \stackrel{ \eqref{def_eq_tree_sphere} }{=}
 \sum_{k=0}^{n}  \sum_{ m' \in T_{(n)}^{m,k}  } g(\mathcal{T}(m), \mathcal{T}(m')) 
 \stackrel{ \eqref{green_bounds}, \eqref{eq_far_in_tree_far_in_embedding} }{\leq} \\
 C_g +  \sum_{k=1}^{n} C_g L_{k-1}^{2-d} \left|  T_{(n)}^{m,k} \right| 
 \stackrel{ \eqref{def:scalesLn}, \eqref{cardinality_tree_sphere} }{ = } C_g \cdot \left( 1+ \sum_{k=1}^{n} 6^{(k-1)(2-d)} 2^{(k-1)}  \right)
 \stackrel{ d \geq 3 }{\leq} \\
  C_g \cdot \left( 1+ \sum_{k=1}^{\infty} 3^{1-k}    \right)
  = \frac{5}{2} C_{g}.
 \end{multline}
 Now \eqref{subcrit_capa_bound} follows from  \eqref{bound_on_capacity}, \eqref{subcrit_upper_bound_on_green_sum}
 and the fact that $ |\mathcal{X}_{\mathcal{T}}|=2^n$. The proof of  Proposition \ref{prop_subcrit} is complete.
\end{proof}

\section{Lower bound on $u_*$}\label{section:lower}

Let us choose $L_0$ according to \eqref{def_eq_L_0_supercrit} in \eqref{def:scalesLn}.
Recall the notion of the plane $F$ from \eqref{plane}.
For $n \geq 1$ and $x \in \mathcal{L}_n \cap F $  let us denote by $B^u_{n,x}$ the event 
\[ B^u_{n,x} = \left\{ 
\begin{array}{cc}
\text{ there exists a 
$*$-connected path in $\mathcal{I}^u \cap F$ } \\
\text{ that connects $S(x, L_n-1)$ to $S(x,2 L_n)$ }
\end{array}
\right\} .
\]

Recall the definitions of $c_g, C_g$ from \eqref{green_bounds} and $\mathcal{C}_d$ from \eqref{def_eq_C_d}.
\begin{proposition}\label{prop_supercrit}
For any $d \geq 3$ and 
\begin{equation}\label{def_eq_lower_bound_on_u_star}
u<  \frac{c_g}{L_0}  \frac{1}{\mathcal{C}_2} 2^{-(d+5)} ,
\end{equation}
 for any $n \geq 1$ and $x \in \mathcal{L}_n \cap F$ we have 
\begin{equation}\label{supercrit_no_star_path_bound}
\mathbb{P}[ B^u_{n,x} ] \leq \left( \frac{3}{4} \right)^{2^n}.
\end{equation}
\end{proposition}

\begin{corollary} Proposition \ref{prop_supercrit} implies the lower bound of Theorem \ref{thm_bounds_on_u_star},
 as we now explain.
Let us denote by $\widehat{A}^u_{n}$ the event that there exists a nearest-neighbour path in $\mathcal{V}^u \cap F$ that connects $S(0, L_n)$ to infinity
 and by 
$\widehat{A}^u_{\infty}$ the event that $\mathcal{V}^u \cap F$ has an infinite connected component. 
By planar duality the event $(\widehat{A}^u_{n})^c$ is equal to the event that there exists a 
$*$-connected path in $\mathcal{I}^u \cap F$ that surrounds $S(0, L_n-1)$,
 thus if \eqref{def_eq_lower_bound_on_u_star} holds, then
\begin{equation*}
 \mathbb{P}[\widehat{A}^u_{n}] \geq
1- 
\mathbb{P}\left[\bigcup_{k=n}^{\infty} \; 
\bigcup_{x \in \mathcal{L}_k, \, |x| \leq 2L_{k+1} }  B^u_{k,x}   \right]
\stackrel{\eqref{def:scalesLn}, \eqref{supercrit_no_star_path_bound} }{\geq}
 1- \sum_{k=n}^{\infty} 25^d \cdot \left( \frac{3}{4} \right)^{2^k},
\end{equation*}
which in turn implies $\mathbb{P}[\widehat{A}^u_{\infty}] = \lim_{n \to \infty} \mathbb{P}[\widehat{A}^u_{n}]=1$.
 Therefore we have $u_* \geq \frac{c_g}{L_0}  \frac{1}{\mathcal{C}_2} 2^{-(d+5)}$.
\end{corollary}

\begin{proof}[Proof of Proposition \ref{prop_supercrit}]
We say that $\mathcal{T}: T_n \to F$ is a  proper embedding of the dyadic tree $T_n$  with 
root at $x \in \mathcal{L}_n \cap F$ into $F$ if $\mathcal{T} \in \Lambda_{n,x}$ (see Definition \ref{def_proper_embedding_of_trees}). We denote by $\Lambda^F_{n,x}$ the set of proper embeddings of  $T_n$ into $F$.

For any $y \in \mathcal{L}_0 \cap F$ let us define the frame $\Box_y \subseteq F$ by
\begin{equation*}
\Box_y \stackrel{\eqref{def_eq_frame_1} }{=} \Box_y^{L_0}= S(y, L_0 -1) \cap F. 
\end{equation*} 

For any $n \geq 1$, $x \in \mathcal{L}_n \cap F$ and $\mathcal{T} \in \Lambda^F_{n,x}$ let us denote by
\begin{equation}\label{def_eq_frame_union}
 \mathcal{X}^{\Box}_{\mathcal{T}}= \bigcup_{m \in T_{(n)}} \Box_{\mathcal{T}(m)}.
 \end{equation}
We start the proof of Proposition \ref{prop_supercrit} by an application 
of Lemma \ref{lemma_path_leaves_spheres_intersect} with $d=2$:
 \begin{multline}\label{supercrit_tree_union_bound}
  \mathbb{P} \left[ B^u_{n,x} \right] 
  \stackrel{ \eqref{leaves_spheres_path_intersect} }{\leq}
\mathbb{P} \left[ \bigcup_{\mathcal{T} \in \Lambda^F_{n,x}} \bigcap_{ m \in T_{(n)} } 
 \{ \Box_{\mathcal{T}(m)} \cap \mathcal{I}^u \neq \emptyset \} \right] 
\stackrel{(*)}{\leq} \\
\mathcal{C}_2^{2^n} \cdot \max_{\mathcal{T} \in \Lambda^F_{n,x}} 
\mathbb{P} \left[  \bigcap_{ m \in T_{(n)} } 
 \{ \Box_{\mathcal{T}(m)} \cap \mathcal{I}^u \neq \emptyset \} \right] ,
 \end{multline}
where in $(*)$ we used Lemma \ref{lemma_cardinality_of_embeddings} to infer
$ |\Lambda^F_{n,x}| \leq \mathcal{C}_2^{2^{n}}$.

In order to bound the probability on the right-hand side of \eqref{supercrit_tree_union_bound} let us fix
some $\mathcal{T} \in \Lambda^F_{n,x}$, recall the
constructive definition of random interlacements from Claim \ref{claim:constructive_interlacement} and
 denote the probability underlying the random objects (i.e., $N_K$ and $(X^j)_{j\geq 1}$)
   introduced in that claim by $\mathrm{P}$ when 
 $K=\mathcal{X}^{\Box}_{\mathcal{T}}$.  
 For a simple  random walk $X$ let us denote by 
 \begin{equation*}
  \mathcal{N}(X) = \sum_{m \in T_{(n)}} \mathds{1} [ \{ X \} \cap \Box_{\mathcal{T}(m)} \neq \emptyset ] 
\end{equation*} 
the number of frames of form $\Box_{\mathcal{T}(m)}, \, m \in T_{(n)}$ that $X$ visits. We can bound
\begin{equation}\label{visits_every_frame_then_many_returns}
\mathbb{P} \left[  \bigcap_{ m \in T_{(n)} } 
 \{ \Box_{\mathcal{T}(m)} \cap \mathcal{I}^u \neq \emptyset \} \right] \leq
 \mathrm{P} \left[ \sum_{j=1}^{N_K} \mathcal{N}(X^j) \geq 2^n \right].
\end{equation}
Our next goal is to stochastically bound  $\mathcal{N}(X)$. 
Recall the definitions of $c_g, C_g$ from \eqref{green_bounds} and
$L_0$ from \eqref{def_eq_L_0_supercrit}.
 Let us define
\begin{equation}\label{def_of_p_geo_paramater} 
p =
\begin{cases} 
12 C_g/c_g \cdot L_0^{3-d}   & \text{ if } \quad d \geq 4, \\
12 C_g/c_g \cdot \frac{1}{1+ \ln(L_0)} & \text{ if } \quad d = 3.
\end{cases}
\end{equation}

For any $m \in T_{(n)}$, $y \in  \Box_{\mathcal{T}(m)}$ we have
\begin{multline}\label{walk_escapes_p}
 P_y[ \{X\} \cap  \mathcal{X}^{\Box}_{\mathcal{T}} \setminus \Box_{\mathcal{T}(m)} \neq \emptyset ]
\stackrel{  \eqref{def_eq_tree_sphere}, \eqref{def_eq_frame_union} }{\leq} 
\sum_{k=1}^n \sum_{ m' \in T_{(n)}^{m,k} } 
P_y[ \{X\} \cap \Box_{\mathcal{T}(m')} \neq \emptyset ] 
\stackrel{ \eqref{green_bounds}, \eqref{green_equilib_entrance_identity},  \eqref{eq_far_in_tree_far_in_embedding}  }{\leq} \\
\sum_{k=1}^n \sum_{ m' \in T_{(n)}^{m,k} }  C_g  L_{k-1}^{2-d} \mathrm{cap}(\Box_{\mathcal{T}(m')})
\stackrel{  \eqref{def:scalesLn}, \eqref{cardinality_tree_sphere}   }{=} 
\sum_{k=1}^n 2^{k-1} C_g L_0^{2-d} 6^{(k-1)(2-d)} \mathrm{cap}(\Box_{0})
\stackrel{ d \geq 3 }{\leq} \\
C_g L_0^{2-d} \mathrm{cap}(\Box_{0}) \sum_{k=1}^{\infty} 3^{1-k}   
\stackrel{\eqref{capa_frame_bound}, \eqref{def_of_p_geo_paramater}   }{\leq}
 p.
\end{multline}

The bound \eqref{walk_escapes_p} together with the 
 strong Markov property of simple random walk imply that 
$ P_{\widetilde e_K} [  \mathcal{N}(X) \geq k ] \leq p^{k-1}$ for any $k \geq 1$. In other words, $\mathcal{N}(X)$ is stochastically dominated
by a geometric random variable with parameter $1-p$, which implies 
$E_{\widetilde e_K} \left[z^{\mathcal{N}(X)} \right] \leq \frac{(1-p)z}{1-pz}$ for any $1 \leq z< \frac{1}{p}$. 
Recalling from Claim \ref{claim:constructive_interlacement} that $N_K$ is Poisson 
 with parameter $u \cdot \mathrm{cap}(K)= u \cdot \mathrm{cap}(\mathcal{X}^{\Box}_{\mathcal{T}})$, 
    for any $1 \leq z< \frac{1}{p}$ we obtain
 \begin{multline*}
  \mathrm{E} \left[z^{\sum_{j=1}^{N_K} \mathcal{N}(X^j)  } \right] =
 \exp \left( u \cdot \mathrm{cap}(\mathcal{X}^{\Box}_{\mathcal{T}})
 \left( E_{\widetilde e_K} \left[z^{\mathcal{N}(X)} \right]-1 \right) \right)
\leq \\
 \exp \left( u \cdot \mathrm{cap}(\mathcal{X}^{\Box}_{\mathcal{T}})
\left( \frac{z-1}{1-pz} \right) \right).
\end{multline*}
We can thus apply the exponential Chebyshev inequality with $z=\frac{1}{2p}$ to bound
\begin{multline*}
 \mathbb{P} \left[ B^u_{n,x} \right] 
 \stackrel{ \eqref{supercrit_tree_union_bound}, \eqref{visits_every_frame_then_many_returns}  }{\leq}
 \mathcal{C}_2^{2^n}  \mathrm{E} \left[\left(\frac{1}{2p}\right)^{\sum_{j=1}^{N_K} \mathcal{N}(X^j)  } \right](2p)^{2^n} 
 \leq \\
   \exp \left( u \cdot \mathrm{cap}(\mathcal{X}^{\Box}_{\mathcal{T}})
\left( \frac{\frac{1}{2p}-1}{1/2} \right) \right)(2p \mathcal{C}_2 )^{2^n} 
\stackrel{\eqref{subadditive} }{\leq} 
 \exp \left( u \cdot \frac{\mathrm{cap}(\Box_{0})}{p} \right)^{2^n}
 (2p \mathcal{C}_2 )^{2^n} 
 \stackrel{\eqref{def_eq_L_0_supercrit},\eqref{def_of_p_geo_paramater}  }{\leq} \\
  \exp \left( u \cdot \mathrm{cap}(\Box_{0}) 2^{d} \mathcal{C}_2 \right)^{2^n}
  2^{-2^n}
   \stackrel{ \eqref{capa_frame_bound} }{\leq}
   \exp \left( u  \frac{L_0}{c_g} 2^{d+3} \mathcal{C}_2 \right)^{2^n} 2^{-2^n}
   \stackrel{ \eqref{def_eq_lower_bound_on_u_star} }{\leq}
    \left( \frac{3}{4} \right)^{2^n} .
\end{multline*}
 This completes the proof of Proposition \ref{prop_supercrit}.
\end{proof}

\section*{Acknowledgements} 
The author thanks  Daniel Valesin and Qingsan Zhu 
for inspiring discussions and an anonymous referee for useful comments on the manuscript.
This paper was written while the author was a postdoctoral fellow of the University of British Columbia.
The work of the author is partially supported by OTKA (Hungarian National Research Fund) grant K100473, the Postdoctoral Fellowship of the Hungarian Academy of Sciences and the Bolyai Research Scholarship of the Hungarian Academy of Sciences.

\end{document}